\documentclass{amsart}
\usepackage{lmodern}
\usepackage{microtype}
\usepackage[english]{babel}
\usepackage{hyperref} 
\usepackage{pgfplots}
\usepgfplotslibrary{fillbetween}
\pgfplotsset{compat=1.17}

\theoremstyle{plain} 
\newtheorem{theorem}{Theorem} 
\newtheorem{lemma}[theorem]{Lemma} 

\theoremstyle{remark}
\newtheorem*{question}{Question}
 
\DeclareMathOperator{\supp}{supp} 
\DeclareMathOperator{\PW}{PW}

\begin{document} 
\title[The Nehari problem for the Paley--Wiener space of a disc]{The Nehari problem for \\ the Paley--Wiener space of a disc} 
\date{\today} 

\author{Ole Fredrik Brevig} 
\address{Department of Mathematics, University of Oslo, 0851 Oslo, Norway} 
\email{obrevig@math.uio.no}

\author{Karl-Mikael Perfekt} 
\address{Department of Mathematical Sciences, Norwegian University of Science and Technology (NTNU), NO-7491 Trondheim, Norway} 
\email{karl-mikael.perfekt@ntnu.no}

\begin{abstract}
	There is a bounded Hankel operator on the Paley--Wiener space of a disc in $\mathbb{R}^2$ which does not arise from a bounded symbol.
\end{abstract}


\maketitle

\section{Introduction}
Let $\mathbb{D}$ be the unit disc in $\mathbb{R}^2$. The Paley--Wiener space $\PW(\mathbb{D})$ is the subspace of $L^2(\mathbb{R}^2)$ comprised of functions $f$ whose Fourier transforms $\widehat{f}$ are supported in $\overline{\mathbb{D}}$. For a tempered distribution $\varphi$ we consider the Hankel operator $\mathbf{H}_\varphi$ defined by the equation
\begin{equation}\label{eq:hankel} 
	\widehat{\mathbf{H}_\varphi f}(\eta) = \int_{\mathbb{D}} \widehat{f}(\xi) \widehat{\varphi}(\xi+\eta)\,d\xi, \qquad \eta \in \mathbb{D},
\end{equation}
on the dense subset of $\PW(\mathbb{D})$ comprised of functions $f$ such that $\widehat{f}$ is smooth and compactly supported in $\mathbb{D}$.

We are interested in the characterization of the symbols $\varphi$ such that $\mathbf{H}_\varphi$ extends by continuity to a bounded operator on $\PW(\mathbb{D})$. If $\varphi$ is in $L^\infty(\mathbb{R}^2)$, then clearly 
\begin{equation}\label{eq:basic} 
	\|\mathbf{H}_\varphi f\|_2 \leq \|f\|_2 \|\varphi\|_\infty. 
\end{equation}
Since $\xi+\eta$ is in $2\mathbb{D}$ whenever $\xi$ and $\eta$ are in $\mathbb{D}$, $\mathbf{H}_\varphi = \mathbf{H}_\psi$ for any $\psi$ such that the restrictions of $\widehat{\psi}$ and $\widehat{\varphi}$ to $2\mathbb{D}$ coincide (as distributions in $2\mathbb{D}$). We thus find that
\begin{equation}\label{eq:bsymbol} 
	\|\mathbf{H}_\varphi\| \leq \inf\big\{\|\psi\|_\infty \,:\, \widehat{\psi}\,\big|_{2\mathbb{D}} = \widehat{\varphi}\,\big|_{2\mathbb{D}}\big\}. 
\end{equation}
We say that the Hankel operator $\mathbf{H}_\varphi$ has a bounded symbol if the quantity on the right hand side of \eqref{eq:bsymbol} is finite. We have just demonstrated that if $\mathbf{H}_\varphi$ has a bounded symbol, then $\mathbf{H}_\varphi$ is bounded. We wish to explore the converse.
\begin{question}
	Does every bounded Hankel operator on $\PW(\mathbb{D})$ have a bounded symbol? 
\end{question}

In the classical one-dimensional setting, where the role of $\mathbb{D}$ is played by the half-line $\mathbb{R}_+ = [0,\infty)$, Nehari~\cite{Nehari57} gave a positive answer to this question. We therefore refer to affirmative answers to analogous questions as Nehari theorems. Our question for $\PW(\mathbb{D})$ was first raised implicitly by Rochberg~\cite[Sec.~7]{Rochberg87}, after he had proved that Nehari's theorem holds for the Paley--Wiener space $\PW(I)$ of a finite interval $I \subseteq \mathbb{R}$.

It was conditionally\footnote{The arguments in \cite{CP21} rely on Nehari's theorem for $\mathbb{R}_+ \times \mathbb{R}_+$ as a black box. It was long believed that the Nehari theorem had been proven in this setting, but a significant flaw was recently observed in the available reasoning. We refer to \cite[Section~10]{HTV21} for a detailed discussion.} shown in \cite{CP21} that the Nehari theorem holds for the Paley--Wiener space $\PW(\mathbb{P})$ of any convex polygon $\mathbb{P}$. However, in view of C. Fefferman's negative resolution \cite{Feff71} of the disc conjecture for the Fourier multiplier of a disc, it would not be surprising to see differing results for $\PW(\mathbb{P})$ and $\PW(\mathbb{D})$.

The main purpose of the present note is to establish the following.
\begin{theorem}\label{thm:main} 
	There is a bounded Hankel operator on $\PW(\mathbb{D})$ which does not have a bounded symbol. 
\end{theorem}

Minor modifications of our proof show that if $\mathbb{P}_n$ is an $n$-sided regular polygon, then the optimal constant in the inequality 
\[\inf\big\{\|\psi\|_\infty \,:\, \widehat{\psi}\,\big|_{2\mathbb{P}_n} = \widehat{\varphi}\,\big|_{2\mathbb{P}_n}\big\} \leq C_n\|\mathbf{H}_\varphi\|_{\PW(\mathbb{P}_n)}\]
satisfies $C_n \geq c_\varepsilon n^{1/2 - \varepsilon}$ for any fixed $\varepsilon  > 0$. Here $c_\varepsilon>0$ denotes a constant which depends only on $\varepsilon$. Conversely, the conditional argument of \cite{CP21} yields that $C_n \leq cn$ for some absolute constant $c > 0$. Analogous estimates for Fourier multipliers associated with polygons were considered in \cite{Cord77}.

Finally, let us remark that Ortega-Cerd\`a and Seip \cite{OCS12} have shown that Nehari's theorem also fails for (small) Hankel operators on the infinite-dimensional torus. However, Helson \cite{Hel06} proved that if the Hankel operator is in the Hilbert--Schmidt class $S_2$, then it is induced by a bounded symbol. We are led to the following.

\begin{question}
	Does every Hankel operator on $\PW(\mathbb{D})$ in $S_2$ have a bounded symbol?
\end{question}

In this context, we mention that Peng~\cite{Peng88} has characterized when $\mathbf{H}_\varphi$ is in the Schatten class $S_p$, for $1 \leq p \leq 2$, in terms of the membership of $\varphi$ in certain Besov spaces adapted to $2\mathbb{D}$. In particular, $\mathbf{H}_\varphi$ is in $S_2$ if and only if
\[\int_{2\mathbb{D}} |\widehat{\varphi}(\xi)|^2 (2-|\xi|)^{3/2}\,d\xi < \infty.\]

\section{Proof of Theorem~\ref{thm:main}}
If the Nehari theorem were to hold for $\PW(\mathbb{D})$, there would by the closed graph theorem exist an absolute constant $C < \infty$ such that 
\begin{equation}\label{eq:contradictme} 
	\inf\big\{\|\psi\|_\infty \,:\, \widehat{\psi}\,\big|_{2\mathbb{D}} = \widehat{\varphi}\,\big|_{2\mathbb{D}}\big\} \leq C\|\mathbf{H}_\varphi\| 
\end{equation}
for every bounded Hankel operator on $\PW(\mathbb{D})$. To prove Theorem~\ref{thm:main}, we will construct a sequence of symbols which demonstrates that no such $C<\infty$ can exist.

We begin with an upper bound for $\|\mathbf{H}_\varphi\|$. Guided by the following lemma, our plan is to construct $\varphi$ such that $\mathbf{H}_\varphi$ admits an orthogonal decomposition. For a symbol $\varphi$, define
\[D_\varphi = \big\{\eta \in \mathbb{D}\,:\, \xi+\eta \in \supp{\widehat{\varphi}} \,\text{ for some }\, \xi \in \mathbb{D}\big\}.\]
\begin{lemma}\label{lem:orthogonal} 
	Suppose that $\varphi = \varphi_1 + \varphi_2$ and that $D_{\varphi_1} \cap D_{\varphi_2} = \emptyset$. Then
	\[\mathbf{H}_\varphi = \mathbf{H}_{\varphi_1} \oplus \mathbf{H}_{\varphi_2}.\]
\end{lemma}
\begin{proof}
	Let $f$ be any function in $\PW(\mathbb{D})$ such that $\widehat{f}$ is smooth and compactly supported in $\mathbb{D}$ . Since $\mathbf{H}_\varphi f = \mathbf{H}_{\varphi_1} f + \mathbf{H}_{\varphi_2} f$ by linearity of the integral \eqref{eq:hankel}, it is sufficient to demonstrate that $\mathbf{H}_{\varphi_1} f \perp \mathbf{H}_{\varphi_2} f$. It follows directly from the definition of the Hankel operator \eqref{eq:hankel} that
	\[\supp{\widehat{\mathbf{H}_{\varphi_1} f}} \subseteq D_{\varphi_1} \qquad \text{and} \qquad \supp{\widehat{\mathbf{H}_{\varphi_2} f}} \subseteq D_{\varphi_2}.\]
	By the assumption that $D_{\varphi_1} \cap D_{\varphi_2} = \emptyset$, we therefore conclude that
	\[\langle \mathbf{H}_{\varphi_1} f, \mathbf{H}_{\varphi_2} f \rangle = \langle \widehat{\mathbf{H}_{\varphi_1} f}, \widehat{\mathbf{H}_{\varphi_2} f} \rangle = 0. \qedhere\]
\end{proof}
In particular, if $D_{\varphi_1} \cap D_{\varphi_2}= \emptyset$, then
\[\|\mathbf{H}_\varphi \| = \max(\|\mathbf{H}_{\varphi_1}\|,\|\mathbf{H}_{\varphi_2}\|).\]

Let us next explain the construction of $\varphi$. Consider a radial smooth bump function $\widehat{b}$ which is bounded by $1$, equal to $1$ on $\frac{1}{2}\mathbb{D}$ and compactly supported in $\mathbb{D}$. For a real number $0<r < 1/2$, set $\widehat{b}_r(\xi) = \widehat{b}(\xi/r)$. Note that
\begin{equation}\label{eq:brest} 
	\|\widehat{b}_r \|_1 \leq \pi r^2. 
\end{equation}
For $j=1,2,\ldots,n$, we let $\widehat{\varphi}_j$ be the function obtained by translating $\widehat{b}_r$ by $2-r$ units in the direction $\theta_j = 2 \pi (j-1)/n$, as measured with respect to the positive $\xi_1$-axis in the $\xi_1\xi_2$-plane. We set
\begin{equation} \label{eq:phidef}
\varphi = \varphi_1 + \varphi_2 + \cdots + \varphi_n.
\end{equation}
Since $0 < r < 1/2$, it is clear that $\supp{\widehat{\varphi}} \subseteq 2\mathbb{D} \setminus \mathbb{D}$. Let $r_0 = 1-\frac{1}{\sqrt{2}}=0.29\ldots$.

\begin{figure}
	\centering
	\begin{tikzpicture}
		\begin{axis}[xmin=-2.1,xmax=12.1,ymin=-2.1,ymax=2.1, x=0.85cm, y=0.85cm, hide axis]		
				
			\def\r{1.1}		
			\def\s{5}		
			\def\rmid{1.5}	
			\def\rr{1.8}	
			
			\pgfmathsetmacro\x{\r/2}
			\pgfmathsetmacro\y{sqrt(1-(\r/2)^2)}
			
			\addplot[mark=*,mark size=1pt] coordinates {(\r,0)};
		
			\addplot[domain=0:361,samples=200,thin,color=gray]
				({cos(x)},{sin(x)});
			\addplot[domain=0:361,samples=200,thin,color=gray,densely dotted]
				({2*cos(x)},{2*sin(x)});
		
			\addplot[domain=(\r-1):1, thin, samples=200, draw opacity=0, name path=c1u] 
				{sqrt(1-(max(\r-x,x))^2)};
			\addplot[domain=(\r-1):1, thin, samples=200, draw opacity=0, name path=c1l] 
				{-sqrt(1-(max(\r-x,x))^2)};
			\addplot[gray!25] 
				fill between[of=c1u and c1l];
		
			\addplot[thin]
				coordinates {(\x,-\y) (0,0) (\x,\y)};
				
			\pgfmathsetmacro\x{\s+\rmid/2}
			\pgfmathsetmacro\y{sqrt(1-(\rmid/2)^2)}
			\pgfmathsetmacro\r{\s+\rmid}
			
			\addplot[mark=*,mark size=1pt] coordinates {(\r,0)};
		
			\addplot[domain=0:361,samples=200,thin,color=gray]
				({\s+cos(x)},{sin(x)});
			\addplot[domain=0:361,samples=200,thin,color=gray,densely dotted]
				({\s+2*cos(x)},{2*sin(x)});
			
			\addplot[domain=(\r-1):(\s+1), thin, samples=200, draw opacity=0, name path=c1u] 
				{sqrt(1-(max(\r-x,x-\s))^2)};
			\addplot[domain=(\r-1):(\s+1), thin, samples=200, draw opacity=0, name path=c1l] 
				{-sqrt(1-(max(\r-x,x-\s))^2)};
			\addplot[gray!25] 
				fill between[of=c1u and c1l];
		
			\addplot[thin]
				coordinates {(\x,-\y) (\s,0) (\x,\y)};
				
			\pgfmathsetmacro\x{2*\s+\rr/2}
			\pgfmathsetmacro\y{sqrt(1-(\rr/2)^2)}
			\pgfmathsetmacro\r{2*\s+\rr}
			
			\addplot[mark=*,mark size=1pt] coordinates {(\r,0)};
		
			\addplot[domain=0:361,samples=200,thin,color=gray]
				({2*\s+cos(x)},{sin(x)});
			\addplot[domain=0:361,samples=200,thin,color=gray,densely dotted]
				({2*\s+2*cos(x)},{2*sin(x)});
			
			\addplot[domain=(\r-1):(2*\s+1), thin, samples=200, draw opacity=0, name path=c1u] 
				{sqrt(1-(max(\r-x,x-2*\s))^2)};
			\addplot[domain=(\r-1):(2*\s+1), thin, samples=200, draw opacity=0, name path=c1l] 
				{-sqrt(1-(max(\r-x,x-2*\s))^2)};
			\addplot[gray!25] 
				fill between[of=c1u and c1l];
		
			\addplot[thin]
				coordinates {(\x,-\y) (2*\s,0) (\x,\y)};
					
		\end{axis}
	\end{tikzpicture}
	\caption{Plots of $D(w)$ and the corresponding disc sector from the proof of Lemma~\ref{lem:rsmall}, for $w=1.1$, $w=1.5$, and $w = 1.8$.}
	\label{fig:sector}
\end{figure}
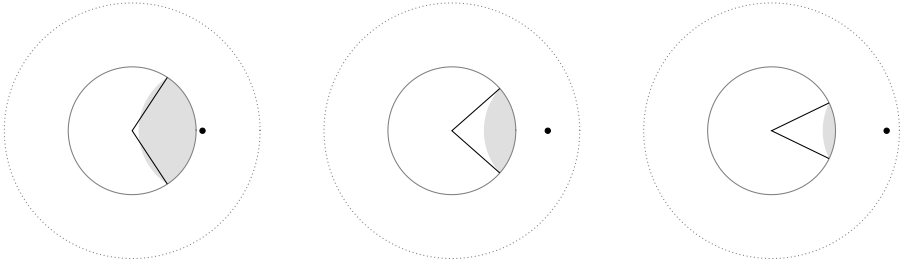

\begin{lemma}\label{lem:rsmall} 
	If $n \geq 2$ and $r=\min(r_0,(2/n)^2)$, then
	\[D_{\varphi_j} \cap D_{\varphi_k} = \emptyset\]
	for every $1 \leq j \neq k \leq n$. 
\end{lemma}
\begin{proof}
	Throughout this proof, we identify $\mathbb{R}^2$ with $\mathbb{C}$. We consider first a simpler situation. For a point $w$ in $2 \mathbb{D} \setminus \mathbb{D}$, let
	\[D(w) = \big\{\eta \in \mathbb{D}\,:\, \xi + \eta = w\, \text{ for some }\, \xi \in \mathbb{D}\big\}.\]
	In other words, $D(w)$ is the intersection of the discs defined by $|\xi|<1$ and $|w-\xi|<1$. To find the intersection of the corresponding circles, we set $\xi = e^{i\theta}$ and let $\theta^{\pm}$ denote the solutions of the equation
	\[1 = |w-e^{i\theta}| \qquad \Longleftrightarrow \qquad \theta^{\pm} = \arg{w} \pm \arccos\left(\frac{|w|}{2}\right).\]
	Let $P_0$ denote the origin, $P_\pm$ the points $e^{i\theta^\pm}$, and $P_w$ the point $w$. The law of cosines implies that the angle $\angle P_0 P_\pm P_w$ is greater than or equal to $\pi/2$ if and only if $|w|\geq \sqrt{2}$. If this holds, then the intersection of the two discs is contained in the disc sector defined by the origin and the two points $P_\pm$. See Figure~\ref{fig:sector}.
	
	Suppose therefore that $|w|\geq \sqrt{2}$ and set $I(w) = (\theta^-,\theta^+)$. If $\xi$ is in $D(w)$, we have just seen that $\arg{\xi}$ is in $I(w)$. It follows that if $w_1$ and $w_2$ are points in $2\mathbb{D} \setminus \sqrt{2}\mathbb{D}$, then 
	\begin{equation}\label{eq:IimpliesS} 
		I(w_1) \cap I(w_2) = \emptyset \qquad \implies \qquad D(w_1) \cap D(w_2) = \emptyset. 
	\end{equation}
	Our goal is now to estimate
	\[I_{\varphi_j} = \bigcup_{w \in \supp{\widehat{\varphi}_j}} I(w).\]
	Since $\supp{\widehat{\varphi}_j}$ is contained in a disc with center $(2-r) e^{i\theta_j}$ and radius $r$, straightforward geometric arguments show that if $w$ is in $\supp{\widehat{\varphi}_j}$, then
	\[|w| \geq 2(1-r) \qquad \text{and} \qquad |\arg{w}-\theta_j| \leq \arctan\left(\frac{r}{2-r}\right).\]
	To ensure that $|w|\geq \sqrt{2}$ we require that $r \leq r_0 = 1-\frac{1}{\sqrt{2}}$. Moreover, if $\theta^\pm$ correspond to the point $w$ as above, then 
	\[|\theta^\pm-\theta_j| \leq \arccos(1-r) + \arctan\left(\frac{r}{2-r}\right) \leq 2 \sqrt{r} + r \leq 3 \sqrt{r}.\]
	Here we used that $2-r \geq 1$ and that $\arctan{r}\leq r$ for $0\leq r \leq 1$. This shows that
	\[I_{\varphi_j} \subseteq \left(\theta_j-3\sqrt{r},\theta_j+3\sqrt{r}\right).\]
	Since $|\theta_j-\theta_k| \geq 2 \pi /n$ for every $1 \leq j \neq k \leq n$ and since $\pi>3$, it follows that if we choose $r = \min(r_0,(\frac{2}{n})^2)$, then we guarantee that $I_{\varphi_j} \cap I_{\varphi_k}= \emptyset$ for every $1 \leq j \neq k \leq n$. The proof is completed by appealing to \eqref{eq:IimpliesS}. 
\end{proof}

Let $\varphi$ be as in \eqref{eq:phidef}, with $n\geq2$ and $r=\min(r_0,(2/n)^2)$. It then follows from Lemma~\ref{lem:orthogonal}, Lemma~\ref{lem:rsmall}, \eqref{eq:basic} and \eqref{eq:brest} that 
\begin{equation}\label{eq:upperbound} 
	\|\mathbf{H}_\varphi\| = \|\mathbf{H}_{\varphi_j}\| \leq \|\varphi_j\|_\infty \leq \|\widehat{\varphi}_j\|_1 = \|\widehat{b}_r \|_1 \leq \pi r^2.
\end{equation}

A lower bound for the left hand side in \eqref{eq:contradictme} will be established through duality.
\begin{lemma}\label{lem:holder} 
	Suppose that $\widehat{f}$ is smooth and compactly supported in $2 \mathbb{D}$. Then
	\[\frac{|\langle \widehat{f}, \widehat{\varphi} \rangle |}{\|f\|_1} \leq \inf\big\{\|\psi\|_\infty \,:\, \widehat{\psi}\,\big|_{2\mathbb{D}} = \widehat{\varphi}\,\big|_{2\mathbb{D}}\big\}.\]
\end{lemma}
\begin{proof}
	Obviously,
	\[\frac{|\langle f, \psi \rangle |}{\|f\|_1} \leq \|\psi\|_\infty,\]
	and when  $\widehat{f}$ is supported in $2\mathbb{D}$ and $\widehat{\psi}|_{2\mathbb{D}} = \widehat{\varphi}|_{2\mathbb{D}}$, we have that
	\[\langle f, \psi \rangle = \langle \widehat{f}, \widehat{\psi} \rangle = \langle \widehat{f},\widehat{\varphi} \rangle. \qedhere\]
\end{proof}

We now need to choose a test function $f$ adapted to the symbol $\varphi$ of \eqref{eq:phidef}. It turns out that  $f=f_1+f_2+\cdots+f_n$, where $f_j = \varphi_j$ for $j=1,2,\ldots,n$, will do. By our choice of $n\geq2$ and $r=\min(r_0,(2/n)^2)$ it is clear that $\supp{\widehat{f}_j} \cap \supp{\widehat{f}_k}=\emptyset$ for every $1 \leq j \neq k \leq n$, since the converse statement would contradict Lemma~\ref{lem:rsmall}. 

Exploiting this, we find that 
\begin{equation}\label{eq:ipest} 
	|\langle f, \varphi \rangle | = \|f\|_2^2 = \|\widehat{f}\|_2^2 = n \|\widehat{b}_r\|_2^2 \geq \frac{\pi}{4} n r^2. 
\end{equation}
To get an upper bound for $\|f\|_1$, we split the integral at some $R>0$,
\[\|f\|_1 = \int_{|x| \leq R} |f(x)|\,dx + \int_{|x| > R} |f(x)|\,dx = I_1 + I_2.\]
For the first integral, we use the Cauchy--Schwarz inequality,
\[I_1 \leq \sqrt{\pi} R \bigg(\int_{|x| \leq R} |f(x)|^2\,dx\bigg)^\frac{1}{2} \leq \sqrt{\pi} R \|f\|_2 = \sqrt{\pi} R \|\widehat{f}\|_2 \leq \pi R \sqrt{n} r,\]
where we again exploited that $\supp{\widehat{f}_j} \cap \supp{\widehat{f}_k} = \emptyset$ for $1 \leq j \neq k \leq n$. For the second integral, we note that $b$ is rapidly decaying, since $\widehat{b}$ is smooth and compactly supported. In particular, for every $\kappa \geq1$ there is a constant $A_\kappa$ such that 
\begin{equation}\label{eq:best} 
	\int_{|x| > \varrho} |b(x)| \,dx \leq \frac{A_\kappa}{\varrho^{\kappa-1}}, 
\end{equation}
holds for every $\varrho>0$. We constructed $\widehat{f_j}$ by translating $\widehat{b}_r$ by $2-r$ units in direction $\theta_j$, so there is a unimodular function $g_j$ such that
\[f_j(x) = g_j(x) b_r(x) = g_j(x) r^2 b(rx).\]
Thus $|f(x)| \leq n r^2 b(rx)$ and \eqref{eq:best}, with $\varrho = Rr$, yields 
\[I_2 \leq n \int_{|x| > R} r^2 |b(rx)|\,dx = n \int_{|x| > rR} |b(x)|\,dx \leq A_\kappa \frac{n}{(Rr)^{\kappa-1}}.\]
Combining our estimates for $I_1$ and $I_2$ and choosing $R = n^{1/(2\kappa)}/r$, we find that 
\begin{equation}\label{eq:f1est} 
	\|f\|_1 = I_1 + I_2 \leq (\pi + A_\kappa) n^{1/2+1/(2\kappa)}. 
\end{equation}
Inserting the estimates \eqref{eq:ipest} and \eqref{eq:f1est} into Lemma~\ref{lem:holder}, we obtain 
\begin{equation}\label{eq:lowerbound} 
	\frac{\pi r^2 n^{1/2-1/(2\kappa)}}{4(\pi + A_\kappa)} \leq \inf\big\{\|\psi\|_\infty \,:\, \widehat{\psi}\,\big|_{2\mathbb{D}} = \widehat{\varphi}\,\big|_{2\mathbb{D}}\big\}. 
\end{equation}

\begin{proof}[Final part of the proof of Theorem~\ref{thm:main}]
	To finish the proof of Theorem~\ref{thm:main}, we combine \eqref{eq:upperbound} and \eqref{eq:lowerbound} to conclude that the constant $C$ in \eqref{eq:contradictme} must satisfy
	\[\frac{n^{1/2-1/(2\kappa)}}{4(\pi+A_\kappa)} \leq C\]
	for any fixed $\kappa\geq1$ and every integer $n\geq2$. Choosing some $\kappa>1$ and letting $n\to \infty$ we obtain a contradiction.
\end{proof}

\bibliographystyle{amsplain} 
\bibliography{discohankel} 

\end{document}